\documentclass{article}
\usepackage[pdftex]{color,graphicx}
\usepackage{amsmath, amsfonts, amsthm, amssymb,latexsym,fancyhdr,mathrsfs,graphics,syntonly,transparent}
\usepackage[colorinlistoftodos]{todonotes}
\usepackage[colorlinks=true,linkcolor=black,citecolor=purple]{hyperref}
\usepackage[all,2cell]{xy}
\usepackage[utf8]{inputenc}

\newcommand{\Cat}[1]{\textbf{\underline{#1}}}
\newcommand{\Lind}[1]{\vphantom{)}_{#1}}

\newtheorem{Thm}{Theorem}[section]
\newtheorem{Lem}{Lemma}[section]
\newtheorem{Corr}{Corrollary}[section]
\newtheorem{Def}{Definition}[section]

\title{N-tuple groups and matched n-tuples of groups}

\author{Dany Majard}

\begin{document}
\maketitle

\begin{abstract}
This paper proves that the category of vacant n-tuple groupoids is equivalent to the category of factorizations of groupoids by n subgroupoids. Moreover it extends this equivalence to the category of maximally exclusive n-tuple groupoids, that we define, and $(n+1)$-factorizations of groupoids with a normal abelian subgroupoid. Finally it shows that in the smooth case, such a factorization gives a presentation of the Poincaré group as a triple groupoid.
\end{abstract}
\tableofcontents

\section*{Introduction}Vacant double groupoids correspond to factorizations of groupoids, as shown in \cite{Andruskiewitsch2009}, and it was proposed by Brown in \cite{Brown11} that a certain notion of vacant triple groupoids would correpond to triple factorizations of groupoids. In this paper, we generalize the definition of core groupoid, prove this result and extend it to all dimensions. Moreover following our previous article \cite{Majard}, we introduce new definitions, that of maximal and exclusive n-tuple groupoids, which allow a wider class of n-tuple groupoids to be analyzed. We show that n-tuple groupoids belonging to this class can embed further factorizations of groupoids and show how the Poincaré group provides a prime example of these.
\clearpage
													        \section{N-tuple categories}
 Let [n] be the set of integers from 1 to n, and $\hat{I}:=[n]\setminus I$. By abuse of notation braces will be omitted in subscripts, for example $n_{ij}:=n_{\{ij\}}$ and $n_{\hat{i}}=n_{[n]\setminus\{i\}}$.
\begin{Def}A \textbf{n-tuple category} $\mathfrak{C}$ is a set
\begin{align*}\bigl(\mathfrak{C}_\emptyset,\{\mathfrak{C}_I\},\{s_I\},\{t_I\},\{\imath_I\}\, \forall I\subseteq\mathbb{Z}_n,\{\circ_i\}\,\forall i\in[n]\bigr)
\end{align*} where :
\begin{itemize}
	\item The sets $\mathfrak{C}_\emptyset$, $\{\mathfrak{C}_I\}\textrm{ for }I\subset[n]$ and $\mathfrak{C}_{[n]}$ are respectively objects, faces and n-cubes.
	\item The following maps of sets are the source and target maps :
		\begin{align*}
			s_I,t_I\,:\,&\mathfrak{C}_J\to \mathfrak{C}_{J\setminus I}\quad\forall J\subseteq[n]\textrm{ s.t. } I\subseteq J
		\end{align*}
	\item The following maps of sets  are the identity maps :
		\begin{align*}
			\imath_I\,:\,&\mathfrak{C}_J\to \mathfrak{C}_{J\cup I}\quad\forall J\subseteq[n]\textrm{ s.t. } I\cap J=\emptyset
		\end{align*}
	\item The following maps of sets are the composition maps.
		\begin{align*}
			\circ_i:\,\mathfrak{C}_I\,\Lind{t_i}\!\!\times_{s_i}\mathfrak{C}_I\to\tau_I\quad\forall I\supset i
		\end{align*}
	\item Sources, targets and compositions are compatible the following way :
			\begin{align*}
				s_Is_J=s_{J\cup I}\qquad\qquad
				t_It_J=t_{J\cup I}
			\end{align*}
	\item Compositions are associative and satisfy the interchange laws :
			\begin{align*}\circ_i\circ_j=\circ_j\circ_i
			\end{align*}
	\item $\imath_i$ is an identity for the composition $\circ_i$.
\end{itemize}
\end{Def}
\begin{Lem}The interchange laws impose $\imath_i\imath_j=\imath_j\imath_i$, and therefore $\imath_I\imath_J=\imath_{J\cup I}$
\end{Lem}
Visually we can represent the elements of a n-tuple category as n-cubes with "oriented faces". The source and targets give the faces, which are themselves i-tuple categories. For example a general element of a double category is a square:
	\begin{figure}[!hbtp]
		\centering
		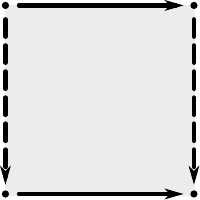
	\end{figure}\\
where arrows of the first type have been represented in bold and arrows of the second type in dashes. Note that there is no defined composition between arrows of different types, so there is no sense in talking of it being commutative. Note as well that there can be many squares with the same boundary.\\
We will take the convention to draw identities arrows as bold segments, regardless of their type.
If two squares can be pasted next to each other in a certain direction, they are composable in that direction, and identities look like:\\\vspace{-3mm}
\begin{figure}[!hbtp]
		\centering
		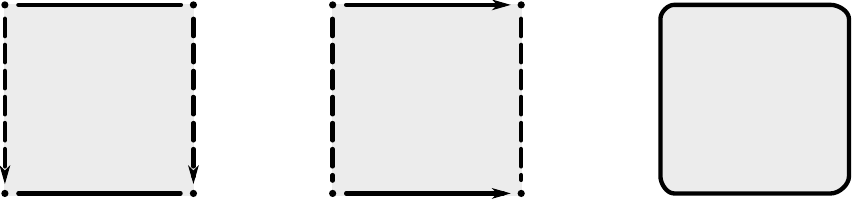
	\end{figure}\\
 The first two are identities on arrows, given respectively by $\imath_1$ and $\imath_2$, whereas the third is an identity on objects, given by $\imath_{12}$.\\
The interchange law ensures that any assortment of the sort :
	\begin{figure}[!hbtp]
		\centering
		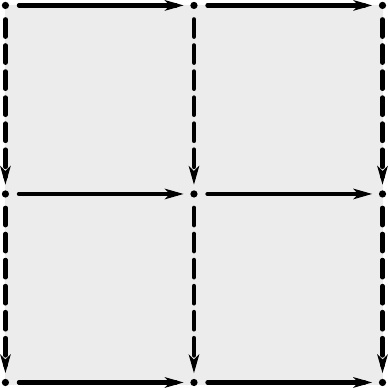
	\end{figure}\\
yields the same square regardless of the order in which it is composed.\\
General elements of a triple category are cubes :
	\begin{figure}[!hbtp]
		\centering
		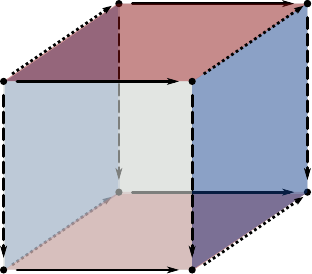
	\end{figure}\\
that compose by pasting in one of the three directions. The interchange laws in each plane ensures that an arrangement that is a barycentric decomposition of a cube yields the same composite, regardless of the order of composition.\\
The generalization to higher dimension is straight forward, though representation on paper becomes challenging.

												        \section{Vacant N-tuple groupoids}
\begin{Def}An \textbf{N-tuple groupoid} is an n-tuple category whose n-cubes are invertible in all directions.
\end{Def}

  The inverses will be denoted as follows, suppose that $X$ is an n-cube, then its inverse with respect to $\circ_i$ will be denoted $X^{-i}$. The interchange laws then ensure that $(X^{-i})^{-j}=(X^{-j})^{-i}$. We can then define the unique inverse $X^{-ijk...}$ of X in the combined directions $i,j,k...$.

				  \subsection{Barycentric subdivisions}

\begin{Def}An arrangement of n-cubes similar to the one given by excluding the subspaces $x_i=\frac{1}{2}$ for all $i \in[n]$ from $[0,1]^n\subset \mathbb{R}^n$ is called a \textbf{barycentric subdivision} of the n-cube. Arrangements given by excluding the subspaces $x_i=1/3$ and $x_i=2/3$ for all $i \in[n]$ is called the \textbf{division in thirds}.
\end{Def}

With this definition in mind we can see the interchange law as ensuring that barycentric subdivisions have a uniquely defined composition.

\begin{Def}In the barycentric subdivision of the n-cube,the \textbf{partition} of a sub n-cube is an ordered pair $(A,B)$ of complementary subsets of $[n]$ such that $i\in A$ if and only if the $i^{th}$ source of the sub n-cube is part of the barycentric division of the $i^{th}$ boundary of the original n-cube. The \textbf{depth} of a sub n-cube is the cardinality of B.
\end{Def}

Then the n-cube adjacent to the source corner has depth 0 and the one adjacent to the sink corner has depth n. Here are for example the 3-subcubes of depth 1:
	\begin{figure}[!hbtp]
		\centering
		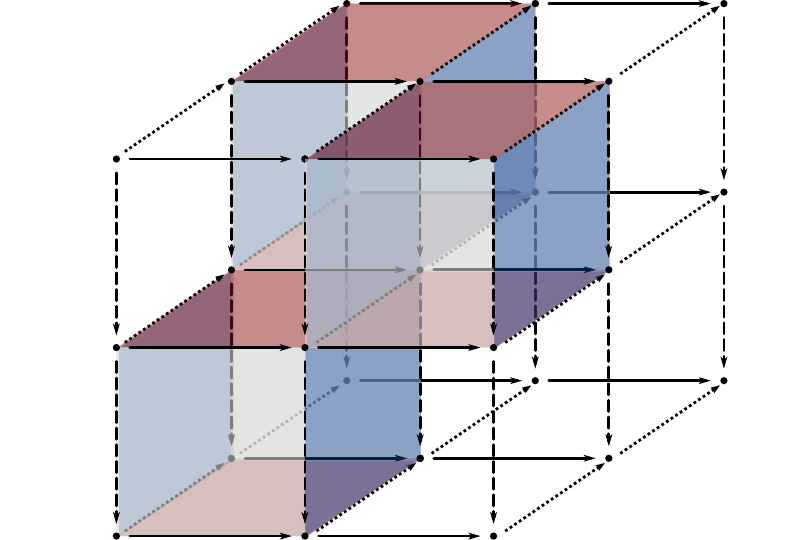
	\end{figure}\\
Every sub n-cube has n external boundaries and n internal boundaries and the $i^{th}$ target of a sub n-cube with partition $(A,B)$ is internal to the subdivision if and only if $i\in A$. Therefore an n-cube has $2n$ boundaries, n of which are internal.
\begin{Lem}A sub n-cube of depth $i$ has a common boundary (n-1)-cube with $i$ sub n-cubes of depth (i-1) and with $(n-i)$ sub n-cubes of depth $(i+1)$.
\end{Lem}
\begin{proof}
Consider a sub n-cube Q, then $t_J(Q)$ is internal to the decomposition if and only if $J\in A_Q$. Similarly, for another sub n-cube Q', $s_J(Q')$ is internal to the decomposition if and only if $J\in B_Q$. Then defining $R:=A_Q\cap B_{Q'}$ we can conclude that $s_R(Q')=t_R(Q)$ and that for $R\subset S\subset[n]$, $s_S(Q')\neq t_S(Q)$. Therefore a sub n-cube $Q$ of depth $i-1$ shares a boundary (n-1)-cube with $Q'$ of depth $i$ if and only if $|A_Q\cap B_{Q'}|=1$, or equivalently $B_Q\subset B_{Q'}$. Since $|B_{Q'}|=|B_Q|+1 =i$ there are $i$ such n-cubes $Q$ for a given $Q'$.
\end{proof}

\begin{Lem}
 The intersection of an n-cube of depth $i$ and the n-cube of depth 0 is an $(n-i)$-cube.  Its intersection with the n-cube of depth n is an $i$-cube. Together, these two intersection contain edges of all n directions.
\end{Lem}
\begin{proof}The sub n-cube $\alpha$ of depth 0 satisfies $A_\alpha=[n]$, so from the previous proof we can conclude that $Q\cap\alpha=s_{B_{Q}}$ and has codimension $|B_{Q}|=i$. Similarly if $\Omega$ is the sub n-cube of depth n, $B_\Omega=[n]$, so the intersection $Q\cap\Omega=t_{A_Q}$ and has dimension $|B_{Q}|=i$.
\end{proof}
Let $X$ be an n-cube of $\tau$ and consider a barycentric subdivision where $X$ has depth n. Place identities in all positions of positive depth. Then the squares that can be placed with depth 0 must have all targets the identity on the source of X, i.e. :
\begin{align*}t_i=\imath_{\hat{i}}(s_{[n]}(X))
\end{align*}

				  	     \subsection{Core groupoids}

\begin{Def}Let $\tau$ be an n-tuple groupoid and define
\begin{align*}\tau_{\lrcorner}:
&=\{\textrm{n-cubes whose recusive targets are identities}\}\\
&=\{X\in\tau|t_i(X)=\imath_{\hat{i}}(t_{[n]}(X))\}
\end{align*}
For $u\in\tau_\lrcorner$  and $X\in\tau$ such that $t_{[n]}(u)=s_{[n]}(X)$ define the \textbf{transmutation of $X$ by $u$}, denoted $u\cdot X$, to be the n-cube accepting a barycentric subdivision with $u$ of depth 0, $X$ of depth 1 and all others identities, as defined above.
\end{Def}
For example, in dimension 3 a element of $\tau_\lrcorner$ looks like the following, where arrows without heads and non colored squares are identities.
\begin{figure}[!htbp]
	\centering
	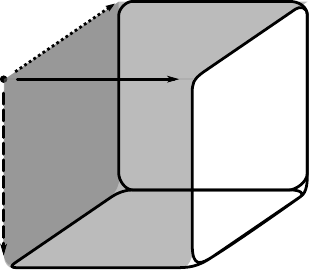
\end{figure}
\begin{Lem} Let $u,v\in \tau_\lrcorner$, then $u\cdot v\in\tau_\lrcorner$. Moreover $(\tau_\lrcorner,\cdot,\imath)$ is a groupoid.
\end{Lem}

\begin{proof}Let source and target of arrows in said groupoid be the source and sinks of the n-cubes of $\tau_\lrcorner$. First note that $t_i(u\cdot X)=t_i(X)$, so that $v\in\tau_\lrcorner$ implies $u\cdot v\in\tau_\lrcorner$. This defines a composition on $\tau_\lrcorner$. The associativity of this composition is a consequence of the interchange law and the uniqueness of identities. The identities given by identity n-cubes on objects. It remains to prove that inverses exist and uniqueness will follow.\\
Consider an equation $b(X)=Y$ where $b(X)$ is a barycentric subdivision with one indeterminate sub-cube $X$, then one can solve for X. This proves the existence of left inverses. Now suppose that $u$ has a right inverse $w$. Then one can isolate $w$ and express it as a composition of n-cubes, containing only identities and an inverse of $u$. This would imply that $w=v$. Solving back for the identity on the source of $u$ shows that v is a right inverse whether $w$ exists or not.
\end{proof}

\begin{Def}The groupoid $\tau_\lrcorner$ is called the \textbf{core groupoid} of $\tau$
\end{Def}
Let $u\in\tau_\lrcorner$, then the assignment $u\to u\cdot$ defines an action of groupoids on $\tau_{[n]}$, the set of n-cubes of $\tau$. The next lemma shows that it is transitive on n-cubes with common targets.
\begin{Lem}Let $X,Y\in\tau$ such that $t_i(X)=t_i(Y),\,\forall i$, then there exists a unique element $u_{XY}\in\tau_\lrcorner$ such that
\begin{align*}X=u_{XY}\cdot Y\end{align*} 
\end{Lem}
\begin{proof}We will proceed in the same fashion as previously. Assume the existence of $u_{XY}$, then using inverses to solve for it gives us a decomposition in terms of $X$, some inverse of $Y$ and identity n-cubes. Yet this composition can be defined without the initial hypothesis, proving the existence of $u_{XY}$.
\end{proof}

\begin{Def} Let $\tau_\bullet$ be the sub groupoid of $\tau_\lrcorner$ composed of n-cubes whose boundaries are all identities. It is called the \textbf{core bundle}
\end{Def}

\begin{Lem}The core bundle is an abelian group bundle over $\tau_0$, the objects of $\tau$.
\end{Lem}
\begin{proof}This is a generalized Eckmann-Hilton argument. N-cubes whose boundaries are identities "slide" along each other. First note that for $u,v\in\tau_\bullet$:
\begin{align*}
	u\cdot v=u\circ_i v\qquad\forall i\in[n]
\end{align*}
To see this, pick $i\in[n]$, $u,v\in\tau_\bullet$ and compose all n-cubes intersecting the $i^{th}$ source of the big n-cube in the barycentric decomposition defining $u\cdot v$. This will give you u since all other subcubes are identities on the source of u. The rest of the subcubes compose to $v$ and therefore, by the interchange law, $u\cdot v=u\circ_i v$.
Now pick two directions $i,j$ and apply the following two dimensional argument in the plane $ij$ :\clearpage
\begin{figure}[!hbtp]
		\centering
		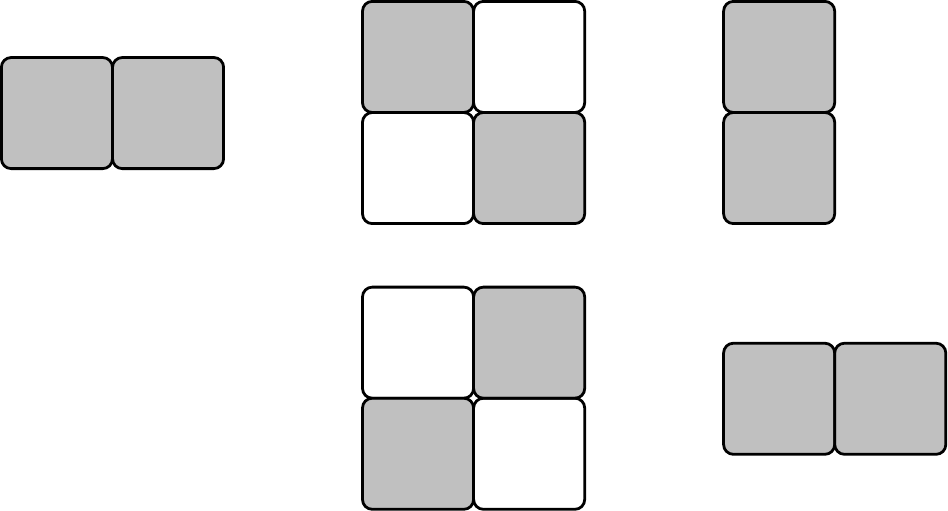
\end{figure}
which shows that $u\cdot v=u\circ_i v = v\circ_i u = v\cdot u$.
\end{proof}

\begin{Corr}Let $X,Y\in\tau$ , then X and Y share all boundaries iff $u_{XY}\in\tau_\bullet$.
\end{Corr}

\begin{Def}A n-tuple groupoid is \textbf{slim} if its core bundle is trivial
\end{Def}

\begin{Corr}A n-tuple groupoid is \textbf{slim} iff there is at most one n-cube per boundary condition.
\end{Corr}

\begin{Def}A n-tuple groupoid is \textbf{exclusive} if $\tau_\lrcorner=\tau_\bullet$.
\end{Def}

A n-cube then belongs to the core groupoid if and only if all its faces are identities.

\begin{Corr}A n-tuple groupoid is exclusive if and only if the boundary of its n-cubes are determined by one of their boundaries of each type.
\end{Corr}
\begin{proof}Suppose that $X,Y\in\tau$ share a boundary of each type and define 
\begin{align*}I:=\{i\in[n]|t_i(X)\neq t_i(Y)\}
\end{align*}
Then $s_i(X)=s_i(Y)\,\forall i\in I$. But in this case $t_j(X^{-I})=t_j(Y^{-I})\,\forall j\in[n]$. But since $\tau$ is exclusive these inverses have the same boundary, proving that X and Y have the same boundaries as well.
\end{proof}

\begin{Lem}Let $\tau$ be a n-tuple groupoid. If all boundary $i^{pl}$ groupoids are slim and all boundary double groupoids are exclusive then $\tau$ is exclusive.\\
Moreover in this case the following is true :
\begin{align*}
	t_{\hat{i}}(X)=t_{\hat{i}}(Y) \quad\forall i\in[n] \iff X=u\cdot Y\,for\,u\in\tau_\bullet
\end{align*}
\end{Lem}

\begin{proof}
Suppose that $\tau_{ij}$ is exclusive for all $i,j\in\mathbb{Z}_n$, then for $X\in\tau_\lrcorner$:
\begin{align*}t_{\hat{i}}(X)&=\imath_i(s_{[n]}(X))\,\forall i\in[n]\\
\Rightarrow\quad s_jt_{\hat{i}\hat{j}}(X)&=\imath_i(s_{[n]}(X))\,\forall i\neq j\in[n]\\
\Rightarrow\quad s_{jk}t_{\hat{i}\hat{j}\hat{k}}(X)&=\imath_i(s_{[n]}(X))\,\forall i\neq j\neq k\in[n]\\
\Rightarrow\cdots\\
\Rightarrow\quad s_{\hat{i}}(X)&=\imath_i(s_{[n]}(X))\,\forall i\in[n]
\end{align*}
Which shows that all 1-arrows of $X$ are identities. Now since the boundary double groupoids of $\tau$ are slim, it shows that all sub 2-cubes of $X$ are identities. Since all boundary triple groupoids of $\tau$ are slim, all sub 3-cubes of $X$ are identities. Repeat the argument to dimension n-1 to prove that $\tau$ is exclusive.\\
In this case, the above argument shows that all boundary i-tuple groupoids are exclusive and slim. Therefore the boundary of an n-cube $X\in\tau$ is fixed by $\{t_{\hat{i}}(X)\}$, so if $X$ and $Y$ share these arrows, they share their whole boundary and by a previous lemma differ by an element of $\tau_\lrcorner$, which in this situation is equal to $\tau_\bullet$
\end{proof}

\begin{Def}An n-tuple groupoid $\tau$ is \textbf{maximal} if for any $(f_1,f_2,\cdots,f_n)$ s.t. $f_i\in\tau_i$ and $t_i(f_i)=t_j(f_j)\quad\forall i,j\in\mathbb{Z}_n$ there exists $X\in\tau$ s.t. $t_{\hat{i}}(X)=f_i$
\end{Def}

\begin{Def}An n-tuple groupoid is \textbf{maximally exclusive} if
\begin{itemize}
	\item all boundary $i^{pl}$ groupoids are slim for $i>1$
	\item all boundary double groupoids are exclusive
	\item it is maximal
\end{itemize}
A n-tuple groupoid is \textbf{vacant} if it is slim and maximally exclusive.
\end{Def}

Let $X$ be an n-cube in a vacant groupoid then $\{t_{\hat{i}}\}_{[n]}$, the set of all boundary 1-arrows targeted at the sink,  determines X uniquely and such an X exists for any possible such combination.

\begin{Lem}Let $\tau$ be a vacant n-tuple groupoid, $X,Y\in\tau$, then there exist a unique n-cube $X\cdot Y$ in $\tau$ that has a barycentric subdivision with $X$ of depth 0 and $Y$ of depth 1. Moreover $(\tau,\cdot)$ is a groupoid.
\end{Lem}
\begin{proof}As mentioned above, all positions of any depth different from 0 or n have an internal edge of each index shared with either X or Y. For example, in the case $n=3$ a 3-cube of depth 1 shares internal edges of two index with the depth 0 3-cube and of the third index with the depth n 3-cube, as in the picture :\vspace{-1mm}
\begin{figure}[!hbtp]
		\centering
		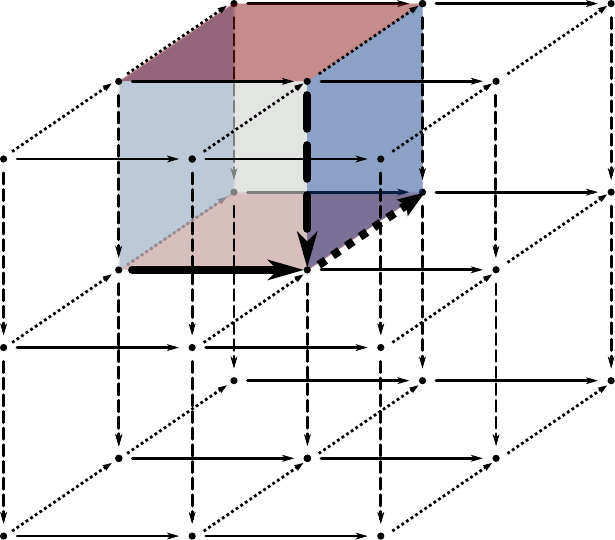
\end{figure}\clearpage
Since the n-tuple groupoid is vacant, these positions have a unique filler which proves that the composition is well defined. Its associativity is guaranteed by uniqueness of fillers and the interchange laws. Identities are identity n-cubes on source or sink.\\
Now to show that inverses exist, consider an n-cube $X$ and place it in position 0. As depth 1 cubes' intersections with X are (n-1) cubes, they only need one arrow in the direction not contained in the shared boundary to be determined. But for X to have an inverse, the n-cube Q of depth 1 whose yet undetermined arrows in direction $i$ needs to satisfy $s_{\hat{i}}(Q)=s_{\hat{i}}(X)^{-1}$. That determines uniquely all depth 1 n-cubes. From a previous lemma depth $i$ cubes share at least two boundary $(n-1)$-cubes with depth $(i-1)$ sub n-cubes, for $i>1$. This fact determines inductively all subcubes of depth greater than 1. Their composition is a n-cube with a boundary arrow of each type being an identity and is therefore $\imath_{[n]}(s_{[n]}(X))$, showing that X has a right inverse.\\
But since $(X\cdot Y)^{-[n]}=X^{-[n]}\cdot Y^{-[n]}$, a right inverse to $X^{-[n]}$ is a left inverse to $X$, proving that X is invertible.
\end{proof}

      \subsection{Equivalence with factorizations of subgroupoids}

Let \Cat{nGpd} be the category of n-tuple groupoids and n-tuple functors and let \Cat{nSub} be the category defined by :
\begin{itemize}
	\item Objects are (n+1) tuples $(G,H_1,H_2,\cdots,H_n)$ where G is a groupoid, $\{H_i\}_{[n]}$ are subgroupoids.
	\item Arrows are functors $f:G\to G'$ such that $f(H_i)\subset H'_i\,\forall i\in[n]$
\end{itemize}
In this subsection we will build an adjunction between the two categories and find subcategories on which the adjunction restricts to an equivalence of categories. First we build a functor $\Gamma:\Cat{nSub}\to\Cat{nGpd}$\\
For a (n+1) tuples $(G,H_1,H_2,\cdots,H_n)$, let $\Gamma(G,H_1,H_2,\cdots,H_n)$ be the n-tuple groupoid defined by:
\begin{itemize}
	\item Objects of $\Gamma(G,H_1,H_2,\cdots,H_n)$ are objects of $G$.
	\item $\Gamma(G,H_1,H_2,\cdots,H_n)_i = H_i$.
	\item n-cubes are commutative cubes, to be defined below.
\end{itemize}
Let $K$ be the set of all possible cubes that would comply with the first two conditions. Each path from the source object to the sink object of an n-cube gives a sequence of composable arrows in G, with each arrow in a different subgroupoid. Each path then defines a map of sets $K\to G$ by composing the sequences. The n-cubes of K that have a constant value under all paths are the cubes of $\Gamma(G,H_1,H_2,\cdots,H_n)$.

\begin{Lem}$\Gamma(G,H_1,H_2,\cdots,H_n)$ is a n-tuple groupoid.
\end{Lem}

\begin{proof}The idea behind this proof is that pasting commutative diagrams along common boundaries produces other commutative diagrams in an associative way.
\end{proof}

Suppose that $F:(G,H_1,H_2,\cdots,H_n)\to (G',H'_1,H'_2,\cdots,H'_n)$ is a subgroupoid preserving functor, then we can define $\Gamma(F)$ as the n-tuple functor such that $\Gamma(F)_i=F|_{H_i}$.
\begin{Thm}Let $I\subset [n]$ have at least two elements, then
\begin{itemize}
	\item $\Gamma(G,H_1,\cdots,H_n)_I$ is slim
	\item $\Gamma(G,H_1,\cdots,H_n)_I$ is exclusive if and only if $\underset{I}{\cap}H_i$ is discrete
	\item $\Gamma(G,H_1,\cdots,H_n)$ is maximal if and only if 
	\begin{align*}H_{\sigma(1)}H_{\sigma(2)}\cdots H_{\sigma(n)}=H_1H_2\cdots H_n\quad\forall\sigma\in S_n\end{align*}
\end{itemize}
\end{Thm}

\begin{proof}By definition a n-cube exists in $\Gamma(G,H_1,\cdots,H_n)$ if and only if it is a commutative diagram in G. It is therefore uniquely defined by its boundary arrows. With the help of identities this proves the first part.\\
For the second, note that by definition:
\begin{align*}X\in\bigl(\Gamma(G,H_1,\cdots,H_n)_I\bigr)\lrcorner\iff s_{\hat{i}}(X)=s_{\hat{j}}(X)\,\forall i,j\in I
\end{align*}
which if $\bigl(\Gamma(G,H_1,\cdots,H_n)_I\bigr)\lrcorner$ is exclusive are all forced to be identities and vice-versa, proving the second part.\\
For the third part, note that by use of inverses, maximality shows that for any composable n-tuples $h_1h_2\cdots h_n$ of $H_1H_2\cdots H_n$, there is an n-cube X with a path source-sink whose $i^{th}$ arrow is $h_i$. But as cube of $\Gamma(G,H_1,\cdots,H_n)$ correspond to elements of $\underset{\sigma\in S_n}{\cap}H_{\sigma(1)}H_{\sigma(2)}\cdots H_{\sigma(n)}$ the third part is proven.
\end{proof}

\begin{Def}Let \Cat{nMatchSub} be the full subcategory of \Cat{nSub} whose objects are n-tuples $(G,H_1,H_2,\cdots,H_n)$ such that $H_i\cap H_j$ is discrete for all $i\neq j$, and \Cat{nMatch} the full subcategory of \Cat{nMatchsub} whose objects satisfy $G=H_1H_2\cdots H_n$.
\end{Def}

\begin{Thm}
 $\Gamma:\Cat{nMatchSub}\to\Cat{nVacant}$ has a left adjoint $\Lambda$ given by $\Lambda(\tau)=((\tau,\cdot),\tau_1,\tau_2,\cdots,\tau_n)$. Moreover $(\Gamma|_\Cat{nMatch},\Lambda)$ is an equivalence of categories.
\end{Thm}
\begin{proof}Let $(G,H_1,\cdots,H_n)\in\Cat{nMatchSub}$ and $\tau\in\Cat{nVacant}$. Then since every element of $(\tau,\cdot)$ has a decomposition as composition of elements of $\tau_i$,
\begin{align*}
	&\Cat{nMatchSub}\bigl(\Lambda(\tau),\,(G,H_1,\cdots,H_2)\bigr)\simeq\\&\{(f_1,\cdots,f_n)|f_i:\tau_i\to H_i\textrm{ and }f_i=f_j  \textrm{ on objects}\}
\end{align*}
moreover $\tau$ has at most one n-cube per acceptable 1-boundary, so an n-tuple functor is fixed by its values on arrows, i.e.
\begin{align*}
	&\Cat{nVacant}\bigl(\tau,\,\Gamma(G,H_1,\cdots,H_n)\bigr)\simeq\\&\{(f_1,\cdots,f_n)|f_i:\tau_i\to H_i\textrm{ and }f_i=f_j  \textrm{ on objects}\}
\end{align*}
proving the adjunction.\\
Now if $(G,H_1,H_2,\cdots,H_n)\in\Cat{nMatch}$, every arrow of G can be written as a composition of arrows of $\{H_i\}$, hence :
\begin{align*}
	\Cat{nMatchSub}\bigl((G,H_1,\cdots,H_2),\,\Lambda(\tau)\bigr)\simeq\\\{(f_1,\cdots,f_n)|f_i:\tau_i\to H_i\textrm{ and }f_i=f_j  \textrm{ on objects}\}\\
	\Cat{nVacant}\bigl(\Gamma(G,H_1,\cdots,H_n),\,\tau\bigr)\simeq\\\{(f_1,\cdots,f_n)|f_i:\tau_i\to H_i\textrm{ and }f_i=f_j  \textrm{ on objects}\}
\end{align*}
which proves the equivalence of categories between \Cat{nMatch} and \Cat{nVacant}
\end{proof}

This theorem allows us to see some decompositions of groups as higher dimensional groups, where dimension is taken in a very categorical sense.
\begin{Corr}Vacant n-tuple groups are in functorial correspondence with matched n-tuples of subgroups.
\end{Corr}

										         \section{Maximally exclusive N-tuple groupoids}

      			       \subsection{From sections to groupoids}
\begin{Def}Let $(\tau_{\hat{1}},\cdots,\tau_{\hat{n}})$ be boundary $(n-1)^{pl}$ groupoids of some n-tuple groupoid $\tau$. Then the \textbf{coarse n-tuple groupoid} $\Box(\tau_{\hat{1}},\cdots,\tau_{\hat{n}})$ is the slim n-tuple groupoid such that $\Box(\alpha_1,\cdots,\alpha_n)_i=\alpha_i$ and an n-cube exists iff its boundary is admissible. The \textbf{frame} $\blacksquare \tau$ of $\tau$ is then the image of the functor :
\begin{align*}
\Pi:\,\tau\to\Box(\tau_{\hat{1}},\cdots,\tau_{\hat{i}},\cdots,\tau_{\hat{n}})
\end{align*}
such that $\Pi s_i= s_i\Pi$ and $\Pi t_i=t_i \Pi$ for all $i\in[n]$.
\end{Def}

In this light, a maximally exclusive n-tuple groupoids is one whose frame is vacant. When trying to define a diagonal composition for n-cubes, we previously used the uniqueness fillers for n-cubes of depth 1 to (n-1) in the barycentric division of the n-cube that vacancy provides. In the present case this uniqueness disappears, though the boundaries of such cubes are fixed. We therefore need to make a consistent choice of fillers to define a diagonal groupoid out of a maximally exclusive n-tuple groupoid.\\
Let $!:\blacksquare\tau\to\tau$ be a section of $\Pi$ as n-tuple graphs and $X,Y\in\tau$, then one can use the section to fill the barycentric subdivision of the n-cube with X of depth 0 and Y of depth n. Denote the composite of the subdivision by $X\cdot_! Y$, and the graph defined by $(\tau_{[n]},s_{[n]},t_{[n]})$ with the above product by $(\tau,\cdot_!)$

\begin{Lem}The following is true :
\begin{align*}!\in\Cat{n-tupleGpd}(\Box\tau,\,\tau)\Rightarrow (\tau,\cdot_!)\in\Cat{Gpd}
\end{align*}
\end{Lem}
\begin{proof}Let $X,Y,Z\in\tau$ such that $t_{[n]}(Z)=s_{[n]}(Y)$ and $t_{[n]}(Y)=s_{[n]}(Z)$. Then $(X\cdot Y)\cdot Z$ and $X\cdot(Y\cdot Z)$ are both equal to the cube obtained by composing the "subdivision in thirds" of the n-cube where $X,Y,Z$ are placed on the diagonal source-sink and all other n-subcubes are filled with elements of the section. Since the frame is vacant, such a filling exists and the composition is associative. Identities on objects are the identities of the groupoid and inverses are given by the following argument:\\
Let X be placed in position of depth 0. From a previous theorem, there exists a filling of the barycentric subdivision with section n-cubes such that all boundaries of the composition are identities on $s_{[n]}(X)$. In other words, $\exists Y\in\blacksquare\tau$ such that $X\cdot !(Y)=u_X\in\tau_\bullet$. Then:
\begin{align*}
X\cdot !(Y)\cdot (u_X)^{-i}&=u_X\cdot(u_X)^{-i}\\
&=u_X\circ_i(u_X)^{-i}\\
&=\imath(s_{[n]}(X))
\end{align*}
Which shows that X has a right inverse. The same procedure shows that it has a left inverse and therefore that $(\tau,\cdot_!)$ is a groupoid.
\end{proof}
      \subsection{Equivalence with factorizations of subgroupoids}

As promised in the introduction, we can extend this result to further decompositions of groupoids. Let \Cat{nSemiSub} be the category defined by :
\begin{itemize}
	\item Objects are (n+2) tuples $(G,A,H_1,H_2,\cdots,H_n)$ where $(G,H_1,\cdots,H_n)\in\Cat{nMatchSub}$, $A\in G$ is an abelian group bundle on the objects of $G$ and $ha(h)^{-1}\in A\qquad \forall h\in H_i\,\forall i\in[n]$.
	\item Arrows are functors $f:G\to G'$ such that $f(H_i)\subset H'_i\quad\forall i\in[n]$ and $f(A)\in A'$.
\end{itemize}
Let \Cat{nSemi} be the full subcategory of \Cat{nSemiSub} where objects  $(G,A,H_1,\cdots,H_n)$ satisfy $G=AH_1H_2\cdots H_n$.\\
Let \Cat{nMaxExcl} be the category whose objects are pairs $(\tau,!)$ and arrows are section preserving n-tuple functors.
Then we can build a functor 
\begin{align*}\tilde\Gamma:\Cat{nSemi}\to\Cat{nMaxExcl}
\end{align*}
by building $\tilde\Gamma(G,A,H_1,\cdots,H_n)$, the n-tuple groupoid whose n-cubes are pairs $(X,a)$ with $X\in\Gamma(G,H_1,\cdots,H_n)$  and $a\in A$ and whose compositions are given by :
\begin{align*}
	(X,a)\circ_i(Y,b)=\bigl(X\circ_i Y,\,a s_{\hat{i}}(X)b(s_{\hat{i}}(X)^{-1})\bigr)
\end{align*}
A direct computation shows that these compositions define an n-tuple groupoid.
together with the section $!:\,\Gamma(G,H_1\cdots,H_n)\to \tilde\Gamma(G,H_1\cdots,H_n)$ given by $!(X)=\bigl(X,\imath(s_{[n]}(X))\bigr)$.
On arrows of \Cat{nSemi}, $\tilde\Gamma$ is given by:
\begin{align*}
\tilde\Gamma(F)(X,a)=\bigl(\Gamma(F)(X),F(a)\bigr)
\end{align*}
Once again a direct computation shows that this defines a functor. We are now ready to state the theorem.
\begin{Thm}The functor $\tilde\Gamma:\Cat{nSemiSub}\to\Cat{nMaxExcl}$ has a left adjoint $\tilde\Lambda$ defined by:
\begin{align*}
\tilde\Lambda(\tau,!)&=\bigl((\tau,\cdot_!),\tau_\bullet, \tau_1,\cdots,\tau_n\bigr)\\
\tilde\Lambda(F)&=F
\end{align*}
Moreover $(\tilde\Gamma|_{\Cat{nSemi}},\tilde\Lambda)$ is an equivalence of categories
\end{Thm}
\begin{proof}
\end{proof}Let $\tau$ be maximally exclusive, then :
\begin{align*}\Cat{nMaxExcl}(\omega,\tau)&\simeq\{(F_0,F_1,\cdots,F_n)|F_0:\omega_\bullet\to\tau_\bullet\textrm{ and }F_i:\omega_i\to\tau_i\}
\end{align*}
Moreover if $(G,A,H_1,\cdots,H_n)\in\Cat{nSemi}$, then:
\begin{align*}
	\Cat{nSemi}\bigl((G,H_1,\cdots,H_n),\,(K,B,L_1,\cdots,L_n)\bigr)&\simeq\\\{(F_0,F_1,\cdots,F_n)|F_0:A\to B\textrm{ and }F_i:H_i\to L_i\}
\end{align*}
Considering that the image of $\tilde\Lambda$ is by definition in \Cat{nSemi}, it is enough to prove the two statements.

      						   \subsection{Examples}
Following the work on Lie double groupoids of MacKenzie \cite{mackenzie-1997}, we can define the notions of Lie n-tuple groupoid. Then we can use decompositions such as the Iwasawa decomposition to present certain Lie groups as n-tuple groups. In our previous paper we presented the Poincaré group as a maximal exclusive double group and hinted towards a possible presentation as a triple groupoid. The previous sections then allows us to make that claim.

\begin{Lem}Every Iwasawa decomposition of a semisimple Lie group gives a presentation of the group as a triple group.
\end{Lem}
\begin{proof}An Iwasawa decomposition gives three subgroups K,A and N such that every element $g$ of $G$ can be uniquely written as $g=kan$ for $k\in K$, $a\in A$ and $n\in N$. The adjunction described above then associates to it a vacant triple groupoid.
\end{proof}

\begin{Lem}Since the Iwasawa decomposition of SO(3,1) is given by the following subgroups:
\begin{align*}
K&:=exp\Bigg\{\begin{bmatrix}0&0&0&0\\0&0&a&b\\0&-a&0&c\\0&-b&-c&0\end{bmatrix}\quad|a,b,c\in\mathbb{R}\Bigg\}\simeq SO(3)
\end{align*}
\begin{align*}
A&:=exp\Bigg\{\begin{bmatrix}0&a&0&0\\a&0&0&0\\0&0&0&0\\0&0&0&0\end{bmatrix}\quad|a\in\mathbb{R}\Bigg\}\simeq SO(1,1)\\
N&:=exp\Bigg\{\begin{bmatrix}0&0&a&b\\0&0&a&b\\a&-a&0&0\\b&-b&0&0\end{bmatrix}\quad|a,b\in\mathbb{R}\Bigg\}\\
\end{align*}
 the Poincaré group has a decomposition of the form :
\begin{align*}Poinc\simeq (KAN) \ltimes \mathbb{R}^4_+
\end{align*}
This decomposition is therefore represented by a maximal exclusive triple group whose core is $(R^4,+)$ and boundary groups are $SO(3)$, $SO(1,1)$ and $N$
\end{Lem}
\begin{proof}The Iwasawa decomposition of the Lorentz group $SO(3,1)$ is a standard computation. The semi-direct product with the translations is a well known feature of Euclidian and Minkowskian isometry groups. The rest follows from the previous theorems.
\end{proof}

										         				        \section*{Conclusion}

The cases presented here are some of the most simple cases of n-tuple groupoids available. Even within these some questions remain unanswered. A precise definition of core diagram has not been given yet for the cases $n>2$ and it seems that it would be a weaker invariant that in the two dimensional case. The classification of the classes of n-tuple groupoids that share the same core diagram has not been found either. Moreover the proper representation theory of these entities has not been discussed anywhere, to our knowledge. Considering that a group as important as the Poincaré group is an example of triple groups it seems to be urgent to take a look at these questions. It is our hope that this quick exposition to the subject matter will encourage further development of higher dimensional group theory and representation theory.
\vspace{10mm}

\textbf{Acknowledgements :} The author would like to thank my PhD thesis advisor Louis Crane and David Yetter for their help and support.

\bibliography{C:/SeveM/__Science/__Thesis/References}{}
\bibliographystyle{plain}

\end{document}